\documentclass[11pt]{amsart} 
\usepackage[curve,v2]{xy}

\usepackage{amssymb,latexsym,amsfonts}

\newcommand{\tensor}{\otimes}

\renewcommand{\H}{\operatorname{Hilb}^2\hspace{-.05in}X}

\newcommand{\Tor}{\operatorname{Tor}}

\newcommand{\wtp}{\widetilde{\P}^n}

\newcommand{\wts}{\widetilde{\Sigma}}

\newcommand{\ses}[3]{0\rightarrow#1\rightarrow#2
   \rightarrow#3\rightarrow0}

\newcommand{\F}{{\mathcal F}}

\newcommand{\E}{{\mathcal E}}

\newcommand{\I}{{\mathcal I}}

\renewcommand{\O}{{\mathcal O}}
\renewcommand{\P}{{\mathbb{P}}}
\newcommand{\PP}{{\mathbb{P}}}

\newcommand{\Z}{{\mathbb{Z}}}

\renewenvironment{proof}{\par \medskip \noindent
{\sc Proof:}}{}

\begingroup
\newtheorem{thm}{Theorem}[section]   
\newtheorem{cor}[thm]{Corollary}     
\newtheorem{lemma}[thm]{Lemma}         
\newtheorem{prop}[thm]{Proposition}  

\newtheorem{conj}[thm]{Conjecture}        

\theoremstyle{definition}
\newtheorem{remark}[thm]{Remark}   

\newtheorem{notterm}[thm]{Notation and Terminology}

\endgroup

\begin{document}

\pagenumbering{arabic}

\title{Equations and syzygies of the first secant variety to a smooth curve}

\author[Peter Vermeire]{Peter Vermeire}

\address{Department of Mathematics, 214 Pearce, Central Michigan
University, Mount Pleasant MI 48859}

\email{verme1pj@cmich.edu}


\date{\today}

\begin{abstract} 
We show that if $C$ is a linearly normal smooth curve embedded by a line bundle of degree at least $2g+3+p$ then the secant variety to the curve satisfies $N_{3,p}$.
\end{abstract}

\maketitle

\section{Introduction}
We work throughout over an algebraically closed field $k$ of characteristic zero.
If $X\subset\P^n$ is a smooth variety, then we let $\Sigma_i(X)$ (or just $\Sigma_i$ if the context is clear) denote the (complete) variety of $(i+1)$-secant $i$-planes.  Though secant varieties are a classical subject, the majority of the work done involves determining the dimensions of secant varieties to well-known varieties.  Perhaps the two most well-known results in this direction are the solution by Alexander and Hirschowitz (completed in \cite{ah}) of the Waring problem for homogeneous polynomials and the classification of the Severi varieties by Zak \cite{zak}.  

More recently there has been great interest, e.g. related to algebraic statistics and algebraic complexity, in determining the equations defining secant varieties (e.g. \cite{ar}, \cite{bcg}, \cite{CGG05a}, \cite{CGG05b}, \cite{CGG07}, \cite{CGG}, \cite{unex}, \cite{cs}, \cite{gss}, \cite{kanev}, \cite{Lan06}, \cite{Lan08}, \cite{LM08}, \cite{LW07}, \cite{LW08}, \cite{sidsul}, \cite{ss}).  In this work, we use the detailed geometric information concerning secant varieties developed by Bertram \cite{bertram}, Thaddeus \cite{thaddeus}, and the author \cite{vermeireflip1} to study not just the equations defining secant varieties, but the syzygies among those equations as well.

It was conjectured in \cite{eks} and it was shown in \cite{ravi} that if $C$ is a smooth curve embedded by a line bundle of degree at least $4g+2k+3$ then $\Sigma_k$ is set theoretically
defined by the $(k + 2)\times (k + 2)$ minors of a matrix of linear forms.  This was extended in \cite{ginensky} where the degree bound was improved to $4g+2k+2$ and it was shown that the secant varieties are \textit{scheme} theoretically cut out by the minors.  It was further shown in \cite{vermeireflip2} that if $X\subset\P^n$ satisfies condition $N_2$ then $\Sigma_1(v_d(X))$ is set theoretically defined by cubics for $d\geq2$.

In \cite{vermeiresecantreg} it was shown that if $C$ is a smooth curve embedded by a line bundle of degree at least $2g+3$ then $\I_{\Sigma_1}$ is $5$-regular, and under the same hypothesis it was shown in \cite{sidver} that $\Sigma_1$ is arithmetically Cohen-Macaulay.  Together with the analogous well-known facts for the curve $C$ itself (\cite{fujita}, \cite{mgreen}, \cite{mumford}) this led to the following conjecture, extending that found in \cite{vermeireflip2}:

\begin{conj}\cite{sidver}
Suppose that $C \subset \P^n$ is a smooth linearly normal curve of degree $d \geq 2g+2k+1+p$, where $p,k\geq0$.  Then 
\begin{enumerate}
\item $\Sigma_k$ is ACM and $\I_{\Sigma_k}$ has regularity $2k+3$ unless $g=0,$ in which case the regularity is $k+2$.  
\item $\beta_{n-2k-1, n+1} = \binom{g+k}{k+1}.$ 
\item $\Sigma_k$ satisfies $N_{k+2,p}$.\qed
\end{enumerate}
\end{conj}

\begin{remark}\label{whatsthatmean}
Recall \cite{EGHP} that a variety $Z\subset\P^n$ satisfies $\operatorname{N}_{r,p}$ if the ideal of $Z$ is generated in degree $r$ and the syzygies among the equations are linear for $p-1$ steps.  Note that the better-known condition $N_p$ \cite{mgreen} implies $N_{2,p}$.
\end{remark}

By the work of Green and Lazarsfeld \cite{mgreen},\cite{laz}, the conjecture holds for $k=0$.  Further, by \cite{fisher} and by \cite{vBH} it holds for $g\leq1$, and by \cite{sidver} parts (1) and (2) hold for $k=1$.  In this work, we show that part (3) holds for $k=1$ (Theorem~\ref{forcurves}).  Some analogous results for higher dimensional varieties can be found in \cite{vermsechigh}.

Our approach combines the geometric knowledge of secant varieties mentioned above with the well-known Koszul approach of Green and Lazarsfeld.  To fix notation, if $L$ is a vector bundle on a smooth curve $C$, then we let $\E_L=d_*(L\boxtimes\O)$, where $d:C\times C\rightarrow S^2C$ is the natural double cover, and if $\F$ is a globally generated coherent sheaf on a variety $X$, then we have the coherent sheaf $M_{\F}$ defined via the exact sequence $\ses{M_{\F}}{\Gamma(X,\F)}{\F}$.   As we will be interested only in the first secant variety for the remainder of the paper, we write $\Sigma$ for $\Sigma_1$.

\section{Preliminaries}

Our starting point is the familiar:

\begin{prop}\label{basic}
Let $C\subset\P^n$ be a smooth curve embedded by a line bundle $L$ of degree at least $2g+3$.  Then $\Sigma$ satisfies $\operatorname{N}_{3,p}$ if and only if $H^1(\Sigma,\wedge^aM_L(b))=0$, $2\leq a\leq p+1$, $b\geq2$.  
\end{prop}

\begin{proof}
Because $L$ also induces an embedding $\Sigma\subset\P^n$, we abuse notation and denote the associated vector bundle on $\Sigma$ by $M_L$.
Letting $F=\oplus\Gamma(\Sigma,\O_{\Sigma}(n))$ and
applying \cite[5.8]{eisenbud} to $\O_{\Sigma}$ gives the exact sequence:
\[
0 \to \Tor_{a-1}(F, k)_{a+b} \to H^1(\Sigma,\wedge^{a} M_L(b)) \to H^1(\Sigma,\wedge^{a}\Gamma(\O(1)) \otimes \O_{\Sigma}(b))
\]
As $\Sigma$ is ACM \cite{sidver}, the term on the right vanishes.  
\nopagebreak \hfill $\Box$ \par \medskip
\end{proof}

\begin{notterm}\label{maps}
Under the hypothesis that $\deg(L)\geq 2g+3$, the reader should keep in mind throughout the following morphisms \cite{vermeireflip1}
\begin{center}
{\begin{minipage}{1.5in}
\diagram
 & & S^2C  \\
& Z=C \times C  \urto^{d=\varphi|_{Z}} \dlto^{\pi_2} \dto^{\hspace{-0.04in} \pi_1 = \pi|_{Z}} \ar@{^{(}->}[r]^{\hspace{.2in}i}  & \wts \uto_{\varphi} \dto^{\pi}  \\
C & C  \ar@{^{(}->}[r]  &\Sigma   
\enddiagram
\end{minipage}}
\end{center}
where 
\begin{itemize}
\item $\pi$ is the blow up of $\Sigma$ along $C$
\item  $i$ is the inclusion of the exceptional divisor of the blow-up
\item $d$ is the double cover, $\pi_i$ are the projections
\item $\varphi$ is the morphism induced by the linear system $|2H-E|$ which gives $\wts$ the structure of a $\P^1$-bundle over $S^2C$; note in particular that $\wts$ is smooth.  
\end{itemize}
We make frequent use of the rank $2$ vector bundle $\E_L=\varphi_*\O(H)=d_*\left(L\boxtimes\O\right)$, and the fact \cite[Proposition 9]{vermeiresecantreg} that $R^i\pi_*\O_{\wts}=H^i(C,\O_C)\tensor\O_C$ for $i\geq1$.
\end{notterm}

\begin{prop}\label{prop: wts trans}
If $C$ is a smooth curve embedded by a line bundle $L$ with $\deg(L)\geq 2g+3$, then $\Sigma$ satisfies $\operatorname{N}_{3,p}$ if and only if $$H^1(\wts,\pi^*\wedge^aM_L(b))\rightarrow H^0(\Sigma,\wedge^aM_L(b)\tensor R^1\pi_*\O_{\wts})$$ is injective for $2\leq a\leq p+1$, $b\geq2$.
\end{prop}

\begin{proof}
This follows immediately from the 5-term sequence associated to the Leray-Serre spectral sequence:
$$0\rightarrow H^1(\Sigma,\wedge^aM_L(b))\rightarrow H^1(\wts,\pi^*\wedge^aM_L(b))\rightarrow H^0(\Sigma,\wedge^aM_L(b)\tensor R^1\pi_*\O_{\wts})$$
and Proposition~\ref{basic}.
\qed
\end{proof}

We will need a cohomological result: 
\begin{lemma}\label{quick}
Let $C\subset\P^n$ be a smooth curve embedded by a line bundle $L$ with $\deg(L)\geq 2g+3$.  Then $H^i(\wts,\O_{\wts}(bH-E))=0$ for $i,b\geq1$.
\end{lemma}

\begin{proof}
Because $C$ is projectively normal we have $H^i(\wtp,\O_{\wtp}(bH-E))=0$ for $i,b\geq1$.  Thus $H^i(\wts,\O_{\wts}(bH-E))=H^{i+1}(\wtp,\O_{\wtp}(bH-E)\tensor\I_{\wts})$, but by \cite[2.4(6)]{sidver}, we know that $H^{i+1}(\wtp,\O_{\wtp}(bH-E)\tensor\I_{\wts})=H^{i+1}(\P^n,\I_{\Sigma}(b))=0$ for $i\geq0$, $b\in\Z$.  
\qed
\end{proof}

\section{Main Result}

We first reinterpret the injection in Proposition~\ref{prop: wts trans} as a vanishing on $\wts$ (Proposition~\ref{prop: wts van}), then on $S^2C$ (Corollary~\ref{onhilb}), and finally on $C\times C$ (Theorem~\ref{forcurves}).

\begin{prop}\label{prop: wts van}
Let $C\subset\P^n$ be a smooth curve satisfying $N_p$ embedded by a line bundle $L$ with $\deg(L)\geq 2g+3$.  Then $\Sigma$ satisfies $\operatorname{N}_{3,p}$ if $H^i(\wts,\pi^*\wedge^{a-1+i}M_L\tensor\O(2H-E))=0$ for $2\leq a\leq p+1$, $i\geq1$.
\end{prop}

\begin{proof}
We use Proposition~\ref{prop: wts trans}.  Consider the sequence on $\wts$
$$\ses{\pi^*\wedge^aM_L(bH-E)}{\pi^*\wedge^aM_L(bH)}{\pi^*\wedge^aM_L(bH)\tensor\O_Z}.$$
We know \begin{eqnarray*}
H^1(Z,\pi^*\wedge^aM_L(bH)\tensor\O_{Z})&=&H^1\left(Z,\left(\wedge^aM_L\tensor L^b\right)\boxtimes \O_C\right)\\
&=&H^1(C,\O_C)\tensor H^0(C,\wedge^aM_L\tensor L^b)\\
&=&H^0(\Sigma,\wedge^aM_L(b)\tensor R^1\pi_*\O_{\wts}).
\end{eqnarray*}
The first equality follows as the restriction of $\pi^*\wedge^aM_L(bH)$ to $Z$ is $\wedge^aM_L(bH) \boxtimes \O_C$.  For the second we use the K\"unneth formula together with the fact that $h^1(C, \wedge^aM_L \otimes L^b)=0$ as $C$ satisfies $N_{p}$ \cite{mgreen}.  The third is the last part of \ref{maps}.

Thus $$h^1(\Sigma,\wedge^aM_L(b))=\operatorname{Rank}\left(H^1(\wts,\pi^*\wedge^aM_L(bH-E))\rightarrow H^1(\wts,\pi^*\wedge^aM_L(bH))\right)$$
and so by Proposition~\ref{prop: wts trans} it is enough to show that $H^1(\wts,\pi^*\wedge^aM_L\tensor\O(bH-E))=0$ for $2\leq a\leq p+1$, $b\geq2$.

From the sequence 
$$\ses{\pi^*\wedge^{a+1}M_L\tensor\O(bH-E)}{\wedge^{a+1}\Gamma\tensor\O(bH-E)}{\pi^*\wedge^aM_L\tensor\O((b+1)H-E)}$$
and the fact (Lemma~\ref{quick}) that $H^i(\wts,\O(bH-E))=0$, we see that $H^1(\wts,\pi^*\wedge^{a}M_L\tensor\O(bH-E))=H^{b-2}(\wts,\pi^*\wedge^{a+b-2}M_L\tensor\O(2H-E))$ for $b\geq2$.
\qed
\end{proof}



\begin{lemma}\label{downtohilb}
Let $C\subset\P^n$ be a smooth curve embedded by a line bundle $L$ with $\deg(L)\geq 2g+3$ and consider the morphism $\varphi:\wts\rightarrow S^2C\subset\P^s$ induced by the linear system $|2H-E|$.  Then $\varphi_*\wedge^aM_L=\wedge^aM_{\mathcal{E_L}}$, and hence $H^i(\wts,\pi^*\wedge^aM_L\tensor\O(2H-E))=H^i(S^2C,\wedge^aM_{\E_L}\tensor \O_{S^2C}(1))$.
\end{lemma}

\begin{proof}
Consider the diagram on $\wts$:
\begin{center}
{\begin{minipage}{1.5in}
\diagram
 &  &  & 0\dto & \\
 & 0\dto & 0\dto & K\dto & \\
0\rto &  \varphi^*M_{\E_L}\dto\rto & \Gamma(S^2C,\E_L)\otimes \O_{\wts}\dto\rto & \varphi^*\E_L\dto\rto & 0 \\
0\rto &  \pi^*M_{L}\dto\rto & \Gamma(C,L)\otimes \O_{\wts} \dto\rto & \pi^*L\dto\rto & 0 \\
 &  K\dto & 0 & 0 &  \\
 & 0 &  &  &
\enddiagram
\end{minipage}}
\end{center}
The vertical map in the middle is surjective as we have $\Gamma(S^2C, \E_L) = \Gamma(\wts,\O(H)) = \Gamma(C\times C,L\boxtimes\O) = \Gamma(C, L)$.  Therefore, surjectivity of the lower right horizontal map and commutativity of the diagram show that the righthand vertical map is surjective. 

Note that $R^i\varphi_*\varphi^*\E_L = \E_L \otimes R^i \varphi_*\O_{\wts}$ by the projection formula and that the higher direct image sheaves $R^i \varphi_*\O_{\wts}$ vanish as $\wts$ is a $\PP^1$-bundle over $S^2C$. For the higher direct images, we have $R^i \varphi_*\pi^*L=0$ as the restriction of $L$ to a fiber of $\varphi$ is $\O(1)$ and hence the cohomology along the fibers vanishes.  From the rightmost column, we see $R^i\varphi_*K=0$. From the leftmost column, we have the sequence
$$\ses{\varphi^*\wedge^aM_{\E_L}}{\pi^*\wedge^aM_{L}}{\varphi^*\wedge^{a-1}M_{\E_L}\tensor K}$$
but as $R^i\varphi_*\left(K\tensor\varphi^*\wedge^{a-1}M_{\E_L}\right)=R^i\varphi_*K\tensor\wedge^{a-1}M_{\E_L}=0$, we have $\varphi_*\wedge^aM_L=\wedge^aM_{\mathcal{E}_L}$.
\qed
\end{proof}

Combining Proposition~\ref{prop: wts van} with Lemma~\ref{downtohilb} yields:

\begin{cor}\label{onhilb}
Let $C\subset\P^n$ be a smooth curve satisfying $N_p$ embedded by a line bundle $L$ with $\deg(L)\geq 2g+3$.
Then $\Sigma$ satisfies $\operatorname{N}_{3,p}$ if $$H^i(S^2C,\wedge^{a-1+i}M_{\E_L}\tensor\O(1))=0$$ for $2\leq a\leq p+1$, $i\geq1$.\qed
\end{cor}

We need a technical lemma, completely analogous to \cite[1.4.1]{laz}.

\begin{lemma}\label{belikerob}
Let $X\subset\P^n$ be a smooth curve embedded by a non-special line bundle $L$ satisfying $N_{2,2}$, let $x_1,\cdots,x_{n-2}$ be a general collection of distinct points, and let $D=x_1+\cdots+x_{n-2}$. Then there is an exact sequence of vector bundles on $X\times X$ $$\ses{L^{-1}(D)\boxtimes L^{-1}(D)(\Delta)}{d^*M_{\E_L}}{\displaystyle \bigoplus_i\left(\O(-x_i)\boxtimes\O(-x_i)\right)}$$
\end{lemma}

\begin{proof}
Choose a general point $x_1\in X$ and consider the following diagram on $X\times X$:
\begin{center}
{\begin{minipage}{1.5in}
\diagram
 & 0\dto & 0\dto & 0\dto & \\
0\rto & d^*M_{\E_{L(-x_1)}}\dto\rto & M_{L(-x_1)}\boxtimes \O\dto\rto & \left(\O\boxtimes L(-x_1)\right)(-\Delta)\dto\rto & 0\\
0\rto &  d^*M_{\E_{L}}\dto\rto & M_L\boxtimes\O \dto\rto & \O\boxtimes L(-\Delta)\dto\rto & 0 \\
0\rto &  \O(-x_1)\boxtimes\O(-x_1)\dto\rto & \O(-x_1)\boxtimes\O\dto\rto & \O\boxtimes\left(L\tensor\O_{x_1}\right) (-\Delta)\dto\rto & 0 \\
 &  0 & 0 & 0 & 
\enddiagram
\end{minipage}}
\end{center}
where the center column comes from \cite[1.4.1]{laz}.  Following just as in that proof, we obtain 
$$\ses{d^*M_{\E_{L(-D)}}}{d^*M_{\E_L}}{\displaystyle \bigoplus_i\left(\O(-x_i)\boxtimes\O(-x_i)\right)}$$
from the left column.  Note however that $d^*M_{\E_{L(-D)}}$ is a line bundle, and one checks that $d^*M_{\E_{L(-D)}}=\wedge^2\E^*_{L(-D)}=L^{-1}(D)\boxtimes L^{-1}(D)(\Delta)$.
\qed
\end{proof}

\begin{thm}\label{forcurves}
Let $C\subset\P^n$ be a smooth curve embedded by a line bundle $L$ with $\deg(L)\geq 2g+p+3$, $p\geq0$.  Then $\Sigma$ satisfies $\operatorname{N}_{3,p}$.
\end{thm}

\begin{proof}
First note \cite{mgreen} that such a curve satisfies $N_{p+2}$.  We verify the condition in Corollary~\ref{onhilb}.
Pulling the sequence on $S^2C$
$$\ses{M_{\E_L}}{\Gamma(S^2C,\E_L)}{\E_L}$$
back to $Z=C\times C$ yields the diagram
\begin{center}
{\begin{minipage}{1.5in}
\diagram
 &  &  & 0\dto & \\
 & 0\dto & 0\dto & K\dto & \\
0\rto &  d^*M_{\E_L}\dto\rto & d^*\Gamma(S^2C, \E_L)\dto\rto & d^*\E_L\dto\rto & 0 \\
0\rto &  M_{L\boxtimes\O}\dto\rto & \Gamma(C,L)\dto\rto & L\boxtimes\O\dto\rto & 0 \\
 &  K\dto & 0 & 0 &  \\
 & 0 &  &  &
\enddiagram
\end{minipage}}
\end{center}
As in \cite{sidver} we have $d_*\O_{Z}=\O_{S^2C}\oplus M$ where $d^*M=\O(-\Delta)$, $d_*K=\E_L\tensor M$, and $K=d^*\wedge^2\E_L\tensor \left(L^*\boxtimes\O\right)=\O\boxtimes L(-\Delta)$.  From the left vertical sequence we have
$$\ses{\wedge^ad^*M_{\E_L}}{\wedge^aM_{L\boxtimes\O}}{\wedge^{a-1}d^*M_{\E_L}\tensor K}$$
and pushing down to $S^2C$ yields
$$\ses{\wedge^aM_{\E_L}\oplus \left(\wedge^aM_{\E_L}\tensor M\right)}{d_*\wedge^aM_{L\boxtimes\O}}{\wedge^{a-1}M_{\E_L}\tensor \E_L\tensor M}$$
Twisting this sequence by $\O_{S^2C}(1)\tensor M^*$ gives
$$\ses{\wedge^aM_{\E_L}(1)\tensor M^*\oplus \wedge^aM_{\E_L}(1)}{\O_{S^2C}(1)\tensor M^*\tensor d_*\wedge^aM_{L\boxtimes\O}}{\wedge^{a-1}M_{\E_L}(1)\tensor \E_L}$$
Since $d^*\O_{\H}(1)\tensor M^*=L\boxtimes L\tensor\O(-\Delta)$, it suffices to show that $$H^i(Z,\wedge^{a-1+i}d^*M_{\E_L}\tensor L\boxtimes L\tensor\O(-\Delta))=0$$ for $2\leq a\leq p+1$, $i=1,2$.

Now, by Lemma~\ref{belikerob} we have exact sequences
$$\ses{\wedge^{r-1}Q\tensor\O(D)\boxtimes\O(D)}{\wedge^rd^*M_{\E_L}\tensor L\boxtimes L\tensor\O(-\Delta)}{\wedge^rQ\tensor L\boxtimes L\tensor\O(-\Delta)}$$
where $Q=\displaystyle \bigoplus_i\left(\O(-x_i)\boxtimes\O(-x_i)\right)$.  

On the right, we have a direct sum of vector bundles of the form $F\boxtimes F(-\Delta)$ where $F$ is a line bundle of degree $\deg(L)-r$.  Thus $H^1$ and $H^2$ of the right side will vanish when $\deg(L)-r\geq 2g+1$.

On the left, we have a direct sum of vector bundles of the form $F\boxtimes F$ where $F$ is a line bundle of degree $n-2-(r-1)=\deg(L)-g-r-1$.  Because $x_1,\cdots,x_{n-2}$ are general, $H^1$ and $H^2$ of the left side will vanish when $\deg(L)-g-r-1\geq g$.  Combining these, we see that $H^i(Z,\wedge^{a-1+i}d^*M_{\E_L}\tensor L\boxtimes L\tensor\O(-E_{\Delta}))=0$ for $2\leq a\leq p+1$, $i=1,2$ as long as $\deg(L)\geq 2g+p+3$.

\qed
\end{proof}

\section{Acknowledgments}
This project grew out of work done together with Jessica Sidman, and has benefited greatly from her insight and input, as well as from her comments regarding preliminary drafts of this work.

\end{document}